\documentclass[12pt]{article}
\usepackage{verbatim}
\usepackage{latexsym}
\usepackage{amsfonts}
\usepackage{amsmath}
\usepackage{amsthm}
\usepackage{url}
\usepackage{authblk}
\usepackage{tikz,subcaption}

\usetikzlibrary{calc,intersections, through,arrows,decorations.markings}

\newtheorem{theorem}{Theorem}
\newtheorem{lemma}[theorem]{Lemma}
\newtheorem{corollary}[theorem]{Corollary}

\theoremstyle{definition}

\newcommand{\rc}[3]{\filldraw[fill opacity=#3] (#1,#2) rectangle (#1+5,#2+3);}
\newcommand{\ar}[3]{\rc{#1}{#2}{#3} \rc{#1+5}{#2-1}{#3}\rc{#1+10}{#2-2}{#3}
\rc{#1+1}{#2+3}{#3} \rc{#1+6}{#2+2}{#3}\rc{#1+11}{#2+1}{#3}
\rc{#1+2}{#2+6}{#3} \rc{#1+7}{#2+5}{#3}\rc{#1+12}{#2+4}{#3}
\rc{#1+3}{#2+9}{#3} \rc{#1+8}{#2+8}{#3}\rc{#1+13}{#2+7}{#3}}

\newcommand{\permrect}[2]{\filldraw[fill opacity=0.6] (#1,#2)--(#1+6,#2)--(#1+6,#2+6)--(#1,#2+6)--cycle;
\draw (#1+0.5,#2+0.5) node [vertex]{};}
\newcommand{\permrectangle}[2]{
\permrect{#1}{#2}
\permrect{#1+40}{#2}
\permrect{#1}{#2+40}
\permrect{#1+40}{#2+40}
\permrect{#1-40}{#2}
\permrect{#1-40}{#2-40}
\permrect{#1}{#2-40}
\permrect{#1-40}{#2+40}
\permrect{#1+40}{#2-40}
}

\begin{document}
\title{Permutations
that separate close elements, and rectangle packings in the torus}
\author[1]{Simon R.\ Blackburn}
\author[2]{Tuvi Etzion\thanks{Supported in part by the Israeli Science Foundation
under grant no. 222/19.}}
\affil[1]{Department of Mathematics, Royal Holloway, University of London, Egham, Surrey TW20 0EX, United Kingdom. \texttt{s.blackburn@rhul.ac.uk}}
\affil[2]{Computer Science Department, Technion–Israel Institute of Technology, Haifa 32000, Israel. \texttt{etzion@cs.technion.ac.il}}

\maketitle

\begin{abstract}
Let $n$, $s$ and $k$ be positive integers. For distinct $i,j\in\mathbb{Z}_n$, define $||i,j||_n$ to be the distance between $i$ and $j$ when the elements of $\mathbb{Z}_n$ are written in a circle. So
\[
||i,j||_n=\min\{(i-j)\bmod n,(j-i)\bmod n\}.
\]
A permutation $\pi:\mathbb{Z}_n\rightarrow\mathbb {Z}_n$ is \emph{$(s,k)$-clash-free} if $||\pi(i),\pi(j)||_n\geq k$ whenever $||i,j||_n<s$. So an $(s,k)$-clash-free permutation moves every pair of close elements (at distance less than $s$) to a pair of elements at large distance (at distance at least $k$). The notion of an $(s,k)$-clash-free permutation can be reformulated in terms of certain packings of $s\times k$ rectangles on an $n\times n$ torus.

For integers $n$ and $k$ with $1\leq k<n$, let $\sigma(n,k)$ be the largest value of $s$ such that an $(s,k)$-clash-free permutation of $\mathbb{Z}_n$ exists. 
Strengthening a recent paper of Blackburn, which proved a conjecture of Mammoliti and Simpson, we determine the value of $\sigma(n,k)$ in all cases.
\end{abstract}

\section{Introduction}
\label{sec:introduction}

Let $n$, $s$ and $k$ be positive integers. As in the abstract, for $i,j\in\mathbb{Z}_n$ we define $||i,j||_n$ to be the distance between $i$ and $j$ when the elements of $\mathbb{Z}_n$ are written in a circle. In other words, $||i,j||_n$ is the Lee metric modulo $n$, and $||i,j||_n=\min\{(i-j)\bmod n,(j-i)\bmod n\}$. Let $\pi:\mathbb{Z}_n\rightarrow\mathbb {Z}_n$ be a permutation of $\mathbb{Z}_n$. An \emph{$(s,k)$-clash} is a pair of distinct elements $i,j\in \mathbb{Z}_n$ such that $||i,j||_n<s$ and $||\pi(i),\pi(j)||_n<k$. (So $i$ and $j$, and their image under $\pi$ are both close.) The permutation $\pi$ is \emph{$(s,k)$-clash-free} if there are no $(s,k)$-clashes.

For integers $n$ and $k$ with $1\leq k<n$, we define $\sigma(n,k)$ be the largest value of $s$ such that an $(s,k)$-clash-free permutation $\pi$ of $\mathbb{Z}_n$ exists. What can be said about $\sigma(n,k)$? This question was first considered by Mammoliti and Simpson~\cite{MammolitiSimpson}, who showed that $\sigma(n,k)\leq \lfloor (n-1)/k\rfloor$ and conjectured that
\begin{equation}
\label{eqn:MS_conjecture}
\sigma(n,k)\in\{\lfloor (n-1)/k\rfloor-1,\lfloor (n-1)/k\rfloor\}
\end{equation}
for all $n$ and $k$. This conjecture was proved by Blackburn~\cite{Blackburn}, who gave an explicit construction of an $(s,k)$-clash-free permutation with $s= \lfloor (n-1)/k\rfloor-1$ for all $n$ and $k$.

As pointed out in~\cite{MammolitiSimpson}, we can think of the problem of constructing $(s,k)$-clash-free permutations as a problem of packing rectangles on a torus as follows. Consider an $n\times n$ array consisting of $n^2$ cells, each cell indexed by a pair of elements in $\mathbb{Z}_n$. We index the cells in cartesian fashion, so for example cell $(0,0)$ is at the lower left corner of the array, and cell $(n-1,0)$ is at the lower right corner of the array. Given a permutation $\pi$ of $\mathbb{Z}_m$, we add $n$ dots to the array in cells $(i,\pi(i))$ for $i\in\mathbb{Z}_n$. (This is sometimes known as the \emph{graph} of the permutation.) We add $s\times k$ rectangles (so rectangles that are $s$ cells wide and $k$ cells high) to the array, by shading cells $(i+x,\pi(i)+y)$ for $0\leq x<s$ and $0\leq y<k$. Then $\pi$ is $(s,k)$-clash-free if and only if these rectangles do not overlap. See Figure~\ref{fig:perm_example} for an example. So we may rephrase the problem of constructing an $(s,k)$-clash-free permutation as finding a (non-overlapping) packing of $n$ rectangles of size $s\times k$ in an $n\times n$ discrete torus, such that no two rectangles share the same lower side or the same left hand side. 

\begin{figure}
\begin{center}
\begin{tikzpicture}[fill=gray!50, scale=0.3,
vertex/.style={circle,inner sep=2,fill=black,draw}]

\clip (0,0) rectangle (40,40);

\draw[dotted] (0,0) grid (40,40);

\draw (0,0) rectangle (40,40);

\permrectangle{0}{0};
\permrectangle{1}{6};
\permrectangle{2}{27};
\permrectangle{3}{33};
\permrectangle{4}{12};
\permrectangle{5}{18};
\permrectangle{6}{39};
\permrectangle{7}{5};
\permrectangle{8}{24};
\permrectangle{9}{30};
\permrectangle{10}{11};
\permrectangle{11}{17};
\permrectangle{12}{36};
\permrectangle{13}{2};
\permrectangle{14}{23};
\permrectangle{15}{29};
\permrectangle{16}{8};
\permrectangle{17}{14};
\permrectangle{18}{35};
\permrectangle{19}{1};
\permrectangle{20}{20};
\permrectangle{21}{26};
\permrectangle{22}{7};
\permrectangle{23}{13};
\permrectangle{24}{32};
\permrectangle{25}{38};
\permrectangle{26}{19};
\permrectangle{27}{25};
\permrectangle{28}{4};
\permrectangle{29}{10};
\permrectangle{30}{31};
\permrectangle{31}{37};
\permrectangle{32}{16};
\permrectangle{33}{22};
\permrectangle{34}{3};
\permrectangle{35}{9};
\permrectangle{36}{28};
\permrectangle{37}{34};
\permrectangle{38}{15};
\permrectangle{39}{21};
\end{tikzpicture}
\end{center}
\caption{When $n=40$, $k=s=6$ and $\pi:\mathbb{Z}_{40}\rightarrow\mathbb{Z}_{40}$ is the permutation mapping $0,1,2,3,4,\ldots$ to $0,6,27,33,12,\ldots$ respectively.}
\label{fig:perm_example}
\end{figure}
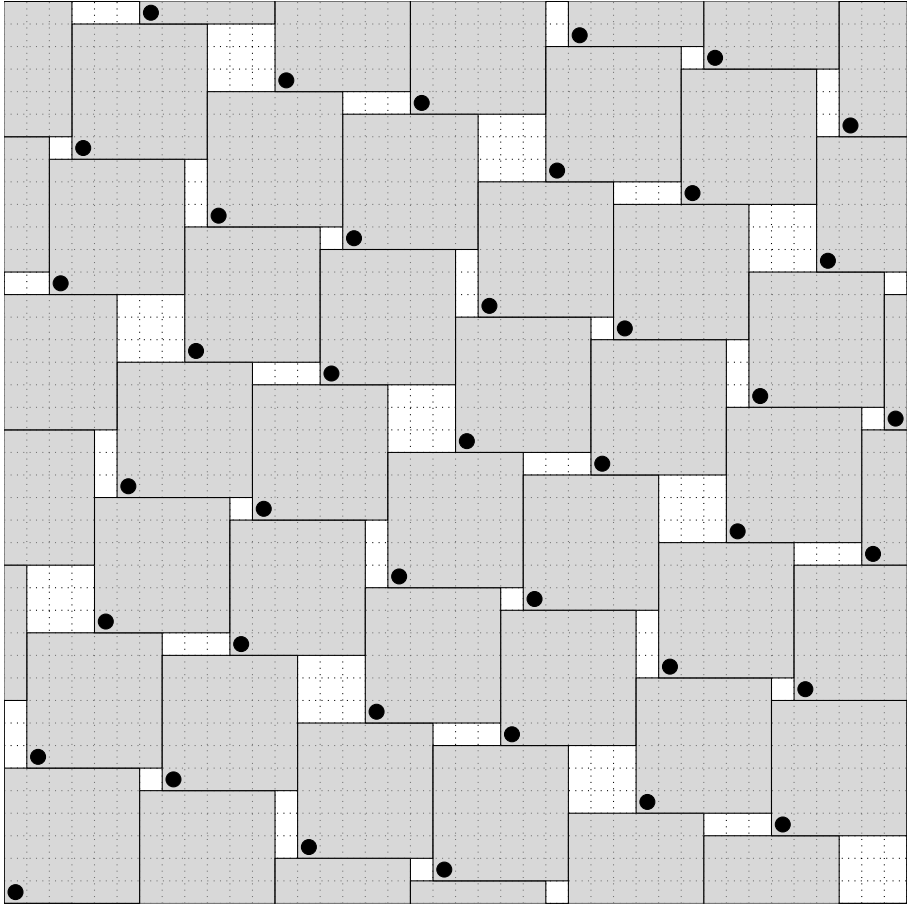

The case $k=2$ of this problem has been considered~\cite{Alspach,BrualdiKiernan,KreherPastine} under the name of cyclic matching sequencibility of graphs. Non-cyclic versions of the problem (in which we work over the integers $\{0,1,\ldots,n-1\}$ rather than modulo $n$) have also been studied~\cite{MammolitiSimpson}. (The special case of the non-cyclic problem when $k=s=2$ can be rephrased as the problem of placing $n$ non-attacking kings on an $n\times n$ chessboard, one in each row and each column. This is the problem of Hertzsprung~\cite{Hertzsprung} from 1887, rediscovered by Kaplansky~\cite{Kaplansky} in 1944.) Packings of diamonds rather than rectangles (called permuted packings) have applications to permutation patterns~\cite{BevanHomberger,BlackburnHomberger}; here we are looking for permutations whose dots have large minimum distance in the Manhattan metric (the $\ell_1$ metric).

Given~\eqref{eqn:MS_conjecture}, a very natural question to ask is: When does $\sigma(n,k)=\lfloor (n-1)/k\rfloor$? Mammoliti and Simpson found many parameters where this is the case. For example, they showed~\cite[Theorem~3.7]{MammolitiSimpson} that, setting $s=\lfloor (n-1)/k\rfloor$, we have $\sigma(n,k)=s$ when
\begin{equation}
\label{eqn:MS conditions}
k|n,\, s|n,\, \gcd(n,k)=1,\text{ or }\gcd(n,s)=1.
\end{equation}
Indeed, using this result and the fact that an $(s,k)$-clash-free permutation is $(s',k')$-clash-free when $s'\leq s$ and $k'\leq k$, they were able to show that $\sigma(n,k)=\lfloor (n-1)/k\rfloor$ for all choices of $n$ and $k$ with $n\leq 30$, apart from $(n,k)\in\{(18,4),(26,4),(26,6)\}$. In these last three cases, they showed by computer search that $\sigma(n,k)=\lfloor (n-1)/k\rfloor-1$. As pointed out in~\cite{Blackburn}, the first open cases with $n>30$ after applying the techniques in~\cite{MammolitiSimpson} are $(n,k)\in\{(34,4),(34,8),(38,6),(39,6),(40,6)\}$.

The main theorem in this paper completely resolves the question of the value of $\sigma(n,k)$, for all values of $n$ and $k$. Indeed, we prove the following theorem:

\begin{theorem}
\label{thm:main}
Let $n$ and $k$ be fixed positive integers, with $k<n$. Write $s=\lfloor (n-1)/k\rfloor$, so $n=sk+r$ where $1\leq r\leq k$. Define $d_k=\gcd(n,k)$ and $d_s=\gcd(n,s)$.
\begin{itemize}
\item If $r\geq s$ or $k=r$, then $\sigma(n,k)=\lfloor (n-1)/k\rfloor$.
\item If $r<s$ and $r<k$ and $d_sd_k$ divides $n$, then $\sigma(n,k)=\lfloor (n-1)/k\rfloor$.
\item If $r<s$ and $r<k$ and $d_sd_k$ does not divide $n$, then $\sigma(n,k)=\lfloor (n-1)/k\rfloor-1$.
\end{itemize}
\end{theorem}

Given~\eqref{eqn:MS_conjecture}, this theorem is equivalent to the assertion that an $(s,k)$-clash-free permutation with $s=\lfloor (n-1)/k\rfloor$ exists if and only if $s\leq r$, $k=r$, or $d_sd_k$ divides $n$.

Considering the small open cases listed above, Theorem~\ref{thm:main} shows in particular (see Figure~\ref{fig:perm_example}) that $\sigma(n,k)=\lfloor (n-1)/k\rfloor$ when $(n,k)=(40,6)$, and $\sigma(n,k)=\lfloor (n-1)/k\rfloor-1$ when $(n,k)\in\{(34,4),(34,8),(38,6),(39,6)\}$.

We observe that the first statement of the theorem follows from known results. To show this, first observe that if $k=r$ then $k$ divides $n$, and so $\sigma(n,k)=\lfloor (n-1)/k\rfloor$ by~\cite[Theorem~3.7]{MammolitiSimpson} and the theorem holds in this case. So we may assume without loss of generality that $1\leq r<k$. Secondly, we observe that if $r>s$ then
\[
\lfloor(n-1)/s\rfloor=\lfloor(ks+(r-1))/s\rfloor\leq k+1
\]
and so $\lfloor (n-1)/s\rfloor-1\leq k$. In this case the construction from Blackburn~\cite[Section~2]{Blackburn} with the roles of $r$ and $s$ swapped gives a $(k,s)$-clash-free permutation of $\mathbb{Z}_n$. Since the existence of an $(s,k)$-clash-free permutation is equivalent to the existence of a $(k,s)$-clash-free permutation,  we find that the theorem holds in this case also. Hence it suffices to assume that $r\leq s$. Finally, if $r=s$ then $s$ divides $r$, and the theorem follows by~\cite[Theorem~3.7]{MammolitiSimpson}. Hence the first assertion of the theorem holds.

So it suffices to prove the final two assertions of the theorem. To this end, we will assume from now on that $n=sk+r$ where $r$ is positive, $r<s$ and $r<k$. 

In Section~\ref{sec:structure}, we will provide some structural information on permutations $\pi$ of $\mathbb{Z}_n$ that are $(\lfloor (n-1)/k,k)$-clash-free. In particular we show (see Theorem~\ref{thm:gap} below) that if an $(s,k)$-clash-free permutation exists then $d_sd_k$ must divide $n$. So if $d_sd_k$ does not divide $n$ then no $(k,s)$-clash-free permutation exists and hence the third statement of the Theorem~\ref{thm:main} will follow. 

In Section~\ref{sec:construction} we assume that $d_sd_k$ divides $n$ and provide an explicit construction for an $(s,k)$-clash-free permutation; see Theorem~\ref{thm:existence} below. This will prove the second statement of the theorem, and completes the proof of Theorem~\ref{thm:main}.

In fact, the techniques of Sections~\ref{sec:structure} and~\ref{sec:construction} can be extended to provide a comprehensive classification of $(s,k)$-clash-free permutations with $s=\lfloor (n-1)/k\rfloor$ when $r<k$, $s<k$ and $d_sd_k$ divides $n$. Section~\ref{sec:morestructure} sketches a proof of this classification.

\paragraph{Acknowledgement} We would like to thank Sergey Kitaev, for bringing the references on non-attacking kings to our attention.

\section{The structure of extremal clash-free permutations}
\label{sec:structure}

Let $n$ and $k$ be such that $k<n$. Define $s=\lfloor (n-1)/k\rfloor$ and write $n=sk+r$ with $1\leq r$. Assume that $r<s$ and $r<k$.
Let $\pi$ be an $(s,k)$-clash-free permutation. This section explores the structure of $\pi$. We will argue from the point-of-view of a rectangle packing for the sake of clarity. So we represent the torus $\mathbb{Z}_n^2$ as an $n\times n$ grid of $1\times 1$ cells, indexed in the usual cartesian fashion with the origin at the lower left-hand corner. We place dots in the cells with coordinates $(i,\pi(i))$. We place $n$ rectangles in the torus, all $s$ cells wide and $k$ cells high, with the $i$th rectangle $R_i$ having cell $(i,\pi(i))$ at its lower left corner. The fact that $\pi$ is $(s,k)$-clash-free implies that these rectangles do not overlap.

\begin{lemma}
\label{lem:rowcolumnfree}
Every column of the grid contains exactly $r$ cells that are not covered by a rectangle. Every row contains exactly $r$ cells that are not covered by a rectangle.
\end{lemma}
\begin{proof}
Let $x\in\mathbb{Z}_n$, and consider the column of all cells with first co-ordinate $x$. The $i$th rectangle $R_i$ intersects our column if and only if $i\in\{x,x-1,\ldots,x-(s-1)\}$. If the rectangle intersects our column, it covers exactly $k$ cells. Since the $s$ rectangles that intersect our column are disjoint, exactly $sk$ cells in our column are covered. So there are $n-sk=r$ cells in our column not covered by a rectangle. The argument for the row of cells with second co-ordinate $y\in\mathbb{Z}_n$ is essentially the same: there are $k$ rectangles intersecting our row (the rectangles $R_i$ where $\pi(i)\in\{y,y-1,\ldots,y-(k-1)\}$) and each rectangle covers $s$ cells in the row.
\end{proof}

A cell is called \emph{free} if it is not contained in a rectangle, and is otherwise \emph{covered}. So the lemma above states that every row and every column contains exactly $r$ free cells.

\begin{lemma}
\label{lem:touching}
Let $R$ be a fixed rectangle. Then exactly one rectangle touches the top edge of $R$. The same is true for the left, right and bottom edges of $R$.
\end{lemma}
\begin{proof}
Consider the $s$ cells that lie just above $R$. All these cells lie in the same row $y$, and this row contains $r$ free cells by Lemma~\ref{lem:rowcolumnfree}. Since $r<s$, not all of the $s$ cells are free. So at least one is occupied, and there is a rectangle touching the top edge of $R$. There cannot be two rectangles touching the top edge of $R$, since then these rectangles would both have dots with second coordinate $y$, contradicting the fact that $\pi$ is injective. So there is exactly one rectangle touching the top edge of $y$, and the first statement of the theorem follows. The same argument (applied three times) establishes the second statement of the theorem. 
\end{proof}

Let $\Gamma$ be the graph with vertex set $\mathbb{Z}_n$, with vertices $i$ and $j$ joined by an edge if and only if the rectangles $R_i$ and $R_j$ touch on their left or right sides. Let $\Delta$ be the graph with the same vertex set $\mathbb{Z}_n$, but with vertices $i$ and $j$ joined by an edge if and only if the rectangles $R_i$ and $R_j$ touch on their top or bottom sides. We call the connected components of $\Gamma$ \emph{warp threads}. We call the connected components of $\Delta$ \emph{weft threads}. (The terminology warp and weft comes from weaving. Parallel warp threads are set up in a loom, and the shuttle with a weft thread attached weaves through the warp threads, thus forming the transverse threads in the woven textile.)

\begin{lemma}
\label{lem:threadsize}
Each warp thread $U$ is a cycle containing $n/d_s$ vertices. The vertices in $U$ correspond to rectangles $R_i$ where $i$ lies in a fixed congruence class modulo $d_s$. Each weft thread $V$ is a cycle containing $n/d_k$ vertices. The vertices in $V$ correspond to rectangles $R_i$ where $\pi(i)$ lies in a fixed congruence class modulo $d_k$ 
\end{lemma}
\begin{proof}
Lemma~\ref{lem:touching} implies that warp threads are all cycles, since every vertex of the thread has degree $2$. The first co-ordinates of rectangles in a warp thread $U$ increases by $s$ as we move right along the cycle. So we return to the start of the cycle after $\ell$ steps, where $\ell$ is the additive order of $s$ modulo $n$. Since $\ell=n/d_s$, the first statement of the theorem follows. The second statement follows from this observation, together with the fact that the number of elements $i\in\mathbb{Z}_n$ in a congruence class modulo $d_s$ is also $n/d_s$. The final statement of the lemma is proved similarly, using the fact that two touching rectangles in a weft thread $V$ have second co-ordinates that differ by $k$. 
\end{proof}

Let $R$ be a rectangle in our packing. There are four rectangles touching $R$, none of which is completely aligned horizontally or vertically with $R$ (because $\pi$ is a permutation). As can be seen in Figure~\ref{fig:clockwise} (where the arrows span a positive number of cells), the fact that rectangles do not overlap shows that the rectangles touching $R$ can be arranged in one of two ways: we call $R$ a \emph{clockwise rectangle} or~\emph{anticlockwise rectangle} respectively. This figure also shows that touching rectangles are clockwise when $R$ is clockwise, and anticlockwise otherwise. So all the rectangles in our packing are either clockwise or anticlockwise: we call $\pi$ a clockwise or anticlockwise permutation depending on which case occurs. Reflecting in a horizontal line (or replacing $\pi$ by $-\pi$) shows that a clockwise $(s,k)$-clash-free permutation exists if and only if an anticlockwise $(s,k)$-clash-free permutation exists.
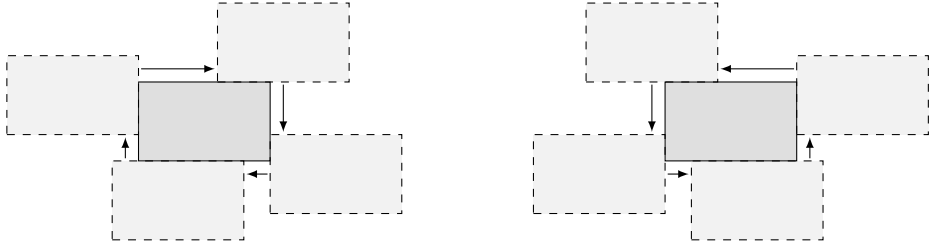
\begin{figure}
\begin{center}
\begin{tikzpicture}[fill=gray!50, scale=0.35]
\filldraw[fill opacity=0.5] (0,0)--(5,0)--(5,3)--(0,3)--cycle;
\filldraw[fill opacity=0.2,dashed] (3,3)--(8,3)--(8,6)--(3,6)--cycle;
\filldraw[fill opacity=0.2,dashed] (-5,1)--(0,1)--(0,4)--(-5,4)--cycle;
\filldraw[fill opacity=0.2,dashed] (5,-2)--(10,-2)--(10,1)--(5,1)--cycle;
\filldraw[fill opacity=0.2,dashed] (-1,-3)--(4,-3)--(4,0)--(-1,0)--cycle;
\draw [-latex] (0.1,3.5) -- (2.9,3.5);
\draw [latex-] (5.5,1.1) -- (5.5,2.9);
\draw [latex-] (4.1,-0.5) -- (4.9,-0.5);
\draw [latex-] (-0.5,0.9) -- (-0.5,0.1);
\filldraw[fill opacity=0.5] (25-0,0)--(25-5,0)--(25-5,3)--(25-0,3)--cycle;
\filldraw[fill opacity=0.2,dashed] (25-3,3)--(25-8,3)--(25-8,6)--(25-3,6)--cycle;
\filldraw[fill opacity=0.2,dashed] (25+5,1)--(25-0,1)--(25-0,4)--(25+5,4)--cycle;
\filldraw[fill opacity=0.2,dashed] (25-5,-2)--(25-10,-2)--(25-10,1)--(25-5,1)--cycle;
\filldraw[fill opacity=0.2,dashed] (25+1,-3)--(25-4,-3)--(25-4,0)--(25+1,0)--cycle;
\draw [-latex] (25-0.1,3.5) -- (25-2.9,3.5);
\draw [latex-] (25-5.5,1.1) -- (25-5.5,2.9);
\draw [latex-] (25-4.1,-0.5) -- (25-4.9,-0.5);
\draw [latex-] (25+0.5,0.9) -- (25+0.5,0.1);
\end{tikzpicture}
\end{center}
\caption{Clockwise (left) and anticlockwise (right) rectangles.}
\label{fig:clockwise}
\end{figure}

We define the \emph{gap} of a rectangle $R$ to be the (edge) connected region of free cells that include free cells just above $R$. 

\begin{lemma}
\label{lem:warpweft}
Fix a warp thread $U$. The set of rectangles touching the upper edges of rectangles in $U$ is another warp thread $U'$. The gap of any rectangle $R\in U$ is rectangular in shape. The width of the gap of $R\in U$ depends only on $U$ (not $R$). The set of rectangles touching the lower edges of rectangles in $U$ form another warp thread. Fix a weft thread $V$. The set of rectangles touching the left edges of rectangles in $V$ forms a weft thread, as does the set of rectangles touching the right edges of rectangles in $V$. The gap of $R\in V$ is a rectangle, and its height depends only on $V$ (not $R$).
\end{lemma}
\begin{proof}
We may assume that $\pi$ is a clockwise permutation. (The proof for an anticlockwise permutation is identical to the clockwise case.) We will prove the statements for warp threads: the corresponding proofs for weft threads are similar.

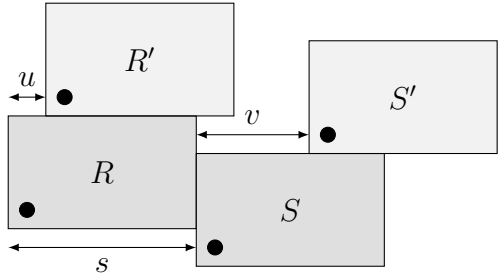
\begin{figure}
\begin{center}
\begin{tikzpicture}[fill=gray!50, scale=0.5,vertex/.style={circle,inner sep=2,fill=black,draw}]
\filldraw[fill opacity=0.5] (0,0) rectangle (5,3);
\draw (2.5,1.5) node {$R$};
\draw (0.5,0.5) node [vertex]{};
\filldraw[fill opacity=0.5] (5,-1) rectangle (10,2);
\draw (7.5,0.5) node {$S$};
\draw (5.5,-0.5) node [vertex]{};
\filldraw[fill opacity=0.2] (1,3) rectangle (6,6);
\draw (3.5,4.5) node {$R'$};
\draw (1.5,3.5) node [vertex]{};
\filldraw[fill opacity=0.2] (8,2)--(13,2)--(13,5)--(8,5)--cycle;
\draw (10.5,3.5) node {$S'$};
\draw (8.5,2.5) node [vertex]{};
\draw [latex-latex] (0,3.5) -- (1,3.5) node[midway,above] {$u$};
\draw [latex-latex] (5,2.5) -- (8,2.5) node[midway,above] {$v$};
\draw [latex-latex] (0,-0.5) -- (5,-0.5) node[midway,below] {$s$};
\end{tikzpicture}
\end{center}
\caption{Can $R'$ and $S'$ fail to touch?}
\label{fig:twowarp}
\end{figure}

Let $R$ and $S$ be adjacent rectangles in a warp thread $U$. Let $R'$ and $S'$ be the rectangles touching the upper sides of $R$ and $S$ respectively; see Figure~\ref{fig:twowarp}. To prove the first statement of the lemma, it suffices to prove that $R'$ and $S'$ lie on the same warp thread. The first coordinates $a$ of the dot in $R$ and $b$ of the dot in $S$ satisfy the equality $b-a=s$, since $R$ and $S$ touch. The first coordinate of the dot in $R'$ is of the form $a'=a+u$ where $1\leq u<s$, and the first coordinate of the dot in $S'$ is $b'=b+v$ where $1\leq v<s$. The number of columns separating $R'$ and $S'$ is $b'-a'-s$, where
\[
b'-a'-s=(b+v)-(a+u)-s=v-u<v<s.
\]
So there can be no rectangle between $R'$ and $S'$, as $R$ and $S$ are separated by fewer than $s$ columns. But this means that $R'$ and $S'$ touch, since otherwise the $k$ cells immediately to the left of $S$ are all free. This implies that $R'$ and $S'$ lie on the same warp thread, as required.

The same argument implies the fourth statement of the lemma, that the rectangles touching the lower sides of a warp thread form another warp thread.

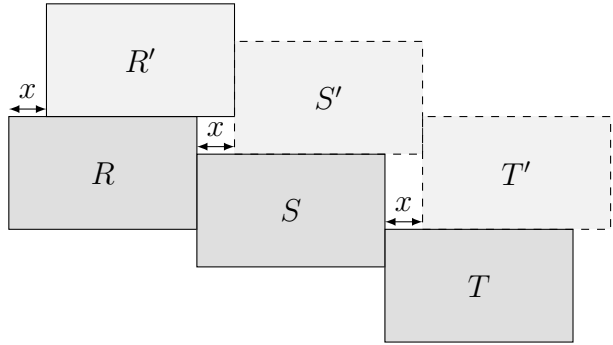
\begin{figure}
\begin{center}
\begin{tikzpicture}[fill=gray!50, scale=0.5,vertex/.style={circle,inner sep=2,fill=black,draw}]
\filldraw[fill opacity=0.5] (0,0) rectangle (5,3);
\draw (2.5,1.5) node {$R$};
\filldraw[fill opacity=0.5] (5,-1) rectangle (10,2);
\draw (7.5,0.5) node {$S$};
\filldraw[fill opacity=0.5] (10,-3) rectangle (15,0);
\draw (12.5,-1.5) node {$T$};
\filldraw[fill opacity=0.2] (1,3) rectangle (6,6);
\draw (3.5,4.5) node {$R'$};
\filldraw[fill opacity=0.2,dashed] (6,2)--(11,2)--(11,5)--(6,5)--cycle;
\draw (8.5,3.5) node {$S'$};
\filldraw[fill opacity=0.2,dashed] (11,0)--(16,0)--(16,3)--(11,3)--cycle;
\draw (13.5,1.5) node {$T'$};
\draw [latex-latex] (0,3.2) -- (1,3.2) node[midway,above] {$x$};
\draw [latex-latex] (5,2.2) -- (6,2.2) node[midway,above] {$x$};
\draw [latex-latex] (10,0.2) -- (11,0.2) node[midway,above] {$x$};
\end{tikzpicture}
\end{center}
\caption{Gaps above a warp thread have constant width $x$.}
\label{fig:constantgap}
\end{figure}
In Figure~\ref{fig:constantgap}, $R$, $S$ and $T$ lie in the same warp thread. Let $R'$, $S'$ and $T'$ be the rectangles touching the upper sides of $R$, $S$ and $T$ respectively. We note that once the position of the rectangles $R$, $S$, $T$ and $R'$ are known, the position of $S'$ is determined since $S'$ touches $S$ and $R'$, and similarly the position of $T'$ are determined since it touches $T$ and $S'$. The figure shows that the gaps of $S$ and $T$ both rectangular, of the same width. This suffices to establish the second and third statements of the lemma.
\end{proof}

We may use the results of Lemma~\ref{lem:warpweft} to define two maps on the set of warp threads, and two maps on the set of all weft threads, as follows. For a warp thread $U$, define $\tau(U)$ to be the warp thread consisting of all rectangles (not in $U$) that touch the upper edges of $U$. Define $\delta(U)$ to be the warp thread consisting of all rectangles (not in $U$) that touch the lower edges of $U$. For a weft thread $V$, define $\rho(V)$ to be the weft thread consisting of all rectangles (not in $V$) that touch the right edges of rectangles in $V$. Finally, define $\lambda(V)$ to be the weft thread consisting of all rectangles (not in $V$) that touch the left edges of rectangles in $V$.
\begin{lemma}
\label{lem:intersection}
The function $\tau$ defined above is a permutation of the set of all $d_s$ warp threads; indeed it is a $d_s$-cycle. The function $\rho$ defined above is a cyclic permutation of the set of weft threads; indeed it is an $d_k$-cycle. 
\end{lemma}
\begin{proof}
The functions $\tau$ and $\rho$ are permutations, as they have inverses $\delta$ and $\lambda$ respectively. Let $U$ be a fixed warp thread. Then the region of the torus consisting of the rectangles in $\bigcup_{i\in\mathbb{Z}}\tau^i(U)$ together with their gaps has no boundary and so is the whole of the torus. Hence every warp thread lies in $\bigcup_{i\in\mathbb{Z}}\tau^i(U)$, and so $\tau$ is an $n/d_s$-cycle. Similarly, $\rho$ is an $n/d_k$-cycle.
\end{proof}

\begin{theorem}
\label{thm:gap}
Let $n$ and $k$ be fixed positive integers, with $k<n$. Write $s=\lfloor (n-1)/k\rfloor$, so $n=sk+r$ where $1\leq r\leq k$. Define $d_k=\gcd(n,k)$ and $d_s=\gcd(n,s)$. Suppose that $r<k$ and $r<s$. If an $(s,k)$-clash-free permutation of $\mathbb{Z}_n$ exists, then $d_sd_k$ divides $n$.
\end{theorem}
\begin{proof}
Fix a rectangle $R$. The rectangle $R$ lies on a warp thread $U_0$ and a weft thread $V$. Journeying upwards along our weft thread $V$, we obtain a sequence $U_0, U_1,\ldots$ of warp threads, where the $i$th rectangle in our journey lies in warp thread $U_i$. The definition of $\tau$ shows that $\tau(U_i)=U_{i+1}$ for $i\geq 0$. By Lemma~\ref{lem:intersection}, $\tau$ is a $d_s$-cycle and so the sequence $U_0,U_1,\ldots$ has exact period $d_s$. But our weft thread contains $n/d_k$ rectangles and so after $n/d_k$ steps our journey reaches $R$ and begins to repeat. So the sequence $U_0,U_1,\ldots$ has period dividing $n/d_k$. Hence $d_s$ divides $n/d_k$ and so $d_sd_k$ divides $n$, as required.
\end{proof}

Previously, the cases when $\sigma(n,k)=\lfloor(n-1)/k\rfloor-1$ were established by computer search (and so only finitely many parameters were known). Theorem~\ref{thm:gap} allows us to construct infinitely many parameters with this property:

\begin{corollary}
\label{cor:bad_parameters}
There are infinitely many choices of parameters $n$ and $k$ such that $\sigma(n,k)=\lfloor(n-1)/k\rfloor-1$.
\end{corollary}
\begin{proof}
Consider, for example, the case when $k\equiv 2\bmod 4$ with $k\geq 10$. Define $s$ to be an integer such that $6\leq s<k$ and $s\equiv 2\bmod 4$. Define $n=sk+s-4$. Note that $\lfloor (n-1)/k\rfloor=s$. In the notation of Theorem~\ref{thm:gap}, we observe that both $d_k$ and $d_s$ are divisible by $2$ but $n$ is not divisible by $4$, so $n$ is not divisible by $d_sd_k$. Theorem~\ref{thm:gap} implies that no $(s,k)$-clash-free permutation exists and so $\sigma(n,k)\not=\lfloor(n-1)/k\rfloor$. Hence $\sigma(n,k)=\lfloor(n-1)/k\rfloor-1$, by~\eqref{eqn:MS_conjecture}.
\end{proof}

\section{An explicit construction of $(s,k)$-clash-free permutations} 
\label{sec:construction}

\begin{theorem}
\label{thm:existence}
Let $n$ and $k$ be fixed positive integers, with $k<n$. Write $s=\lfloor (n-1)/k\rfloor$, so $n=sk+r$ where $1\leq r\leq k$. Define $d_k=\gcd(n,k)$ and $d_s=\gcd(n,s)$. Suppose that $r<k$ and $r<s$. If $d_sd_k$ divides $n$, then there exists an $(s,k)$-clash-free permutation. 
\end{theorem}

\begin{proof}
Suppose that $d_sd_k$ divides $n$. We define a function $\pi:\mathbb{Z}_n\rightarrow\mathbb{Z}_n$ as follows. Let $x\in\mathbb{Z}_n$. We claim there exist unique integers $j\in\{0,1,\ldots,d_s-1\}$, $i\in\{0,1,\ldots,d_k-1\}$ and $\alpha\in\{0,1,\ldots ,n/(d_sd_k)-1\}$ such that $x=(\alpha d_k+i)s+j\bmod n$. To see this, we first use the quotient and remainder theorem to write $x=qd_s+j\bmod n$ for unique integers $q\in\{0,1,\ldots,n/d_s-1\}$ and $j\in\{0,1,\ldots,d_s-1\}$. Since $\{qd_s\bmod n: q\in\mathbb{Z}\}=\{fs\bmod n: f\in\mathbb{Z}\}$, we see that $qd_s=fs\bmod n$ for a unique integer $f\in\{0,1,\ldots,n/d_s-1\}$. We write $f=\alpha d_k+i$, where $0\leq i<d_k$ and $0\leq \alpha<n/(d_sd_k)$. Then $x=(\alpha d_k+i)s+j\bmod n$, and our claim follows. We define
\[
\pi(x)=\pi((\alpha d_k+i)s+j)=jk-i-\alpha d_sd_k\in\mathbb{Z}_n.
\]

We first show that $\pi$ is a permutation. It suffices to prove that $\pi$ is injective. Suppose
\begin{equation}
\label{eqn:x}
x=(\alpha d_k+i)s+j \text{ and }x'=(\alpha' d_k+i')s+j'
\end{equation}
are elements of $\mathbb{Z}_n$ such that $\pi(x)=\pi(x')$. We may look at $\pi(x)=jk-i-\alpha d_sd_k$ and $\pi(x')=j'k-i'-\alpha' d_sd_k$ modulo $d_k$, since $d_k$ divides $n$. Using the fact that $d_k$ also divides $k$, we see that $i=i'$ modulo $d_k$ and so $i=i'$. Hence
\begin{equation}
\label{eqn:alpha}
j(k/d_k)-\alpha d_s=j'(k/d_k)-\alpha' d_s\bmod n/d_k.
\end{equation}
Since $d_s$ divides $n/d_k$, we may deduce that
\begin{equation}
\label{eqn:j}
j(k/d_k)=j'(k/d_k)\bmod d_s.
\end{equation}
Since $d_k=\gcd(n,k)$, we have that $k/d_k$ and $n/d_k$ are coprime. But $d_s$ divides $n/d_k$, and so $k/d_k$ and $d_s$ are coprime. Hence~\eqref{eqn:j} implies that $j=j'\bmod d_s$ and so $j=j'$. Now~\eqref{eqn:alpha} implies that $\alpha=\alpha'\bmod{n/(d_sd_k)}$ and so $\alpha=\alpha'$. We have shown that $i=i'$. $j=j'$ and $\alpha=\alpha'$. Hence $x=x'$ by~\eqref{eqn:x}. Hence $\pi$ is injective, as required.

We now show that $\pi$ is $(s,k)$-clash-free. Define integers $a$ and $b$ by $a=r-(d_s-1)$ and $b=d_sd_k-(d_k-1)$. We claim that $0<a<s$ and $0<b<k$. Clearly $b$ is positive. To prove the other inequalities in our claim, we first note that  $r=n-sk$ and $d_sd_k$ divides $n$ and so $r$ is a (positive) multiple of $d_sd_k$. In particular, we see that $a\leq r$ and $b\leq r$. Since $r<s$ and $r<k$, we see that $a<s$ and $b<k$. Finally, since $r\geq d_sd_k$ we see that $a$ is positive, and our claim follows.

Consider the non-overlapping collection of $s\times k$ rectangles in the infinite grid $\mathbb{Z}^2$, as depicted in Figure~\ref{fig:rectanglesinplane}. This packing of rectangles is divided into $d_k\times d_s$ \emph{blocks} of rectangles (with different shadings in the figure). There are $d_s$ rows and $d_k$ columns of rectangles in each block. The gaps above the rectangles are all of width and height $1$, apart from the gaps above the rectangles in the top row of each block which have width $a$, and the gaps above the the left column of each block, which have height $b$. (This is a special case of a packing obtained from two bi-infinite sequences $(a_i)$ and $(b_j)$ of positive integers, with $a_i\leq s$ and $b_i\leq k$, by defining the gap above the $(i,j)$th rectangle to have a gap of width $a_i$ and height $b_j$.)

Fix one of the blocks of rectangles (coloured more darkly in Figure~\ref{fig:rectanglesinplane}), and choose coordinates so that the lower left corner of the lower left rectangle in the block is $(0,0)$. (This is marked as a dot in the most darkly shaded block.) The set of lower left dots in all blocks (the dots in the figure) form a lattice $L$, where
\begin{align*}
L&=\{\alpha u+\beta v:\alpha,\beta\in\mathbb{Z}\},\text{ with}\\
u&=(d_ks,-(d_k-1)-b)=(d_ks,-d_sd_k)\text{ and}\\
v&=((d_s-1)+a,d_sk)=(r,d_sk).
\end{align*}
(If we take a tile consisting of the rectangles in a fixed block and their associated gaps, the images of this tile under $L$ form a lattice tiling of the plane.) 

Let $\cal P$ be the set of coordinates of the lower left corners of all rectangles (not all blocks) in the packing. So
\[
{\cal P}=\{w+(is+j,jk-i):0\leq i\leq d_k-1,0\leq j\leq d_s-1\text{ and }w\in L\}.
\]
Rectangles in the same block are associated with the same element $w\in L$. 
It is clear that the transformation of the plane defined by $x\mapsto x+w$ are symmetries of our packing for any $w\in L$. We note that $(n,0),(0,n)\in L$, since
\begin{align*}
(n,0)&=\left(\frac{k\,}{d_k}\right)u+v,\text{ and}\\
(0,n)&=-\left(\frac{r}{d_sd_k}\right) u+\left(\frac{s\,}{d_s}\right)v.
\end{align*}
So our rectangle packing is periodic of period dividing $n$ in both horizontal and vertical directions, and therefore gives a well-defined packing of the torus $\mathbb{Z}_n^2$ with $s\times k$ rectangles. But the permutation $\pi$ gives rise to rectangles in this packing, since the $s\times k$ rectangles associated with $\pi$ are at coordinates
\[
(x,\pi(x))=\big((\alpha d_k+i)s+j,jk-i-\alpha d_sd_k\big)=\alpha u+(is+j,jk-i)
\]
modulo $n$, which all lie in $\cal P$. Since $\pi$ is a permutation, these rectangles are distinct modulo $n$ and hence non-overlapping. Hence $\pi$ is $(s,k)$-clash-free, as required.
\end{proof}

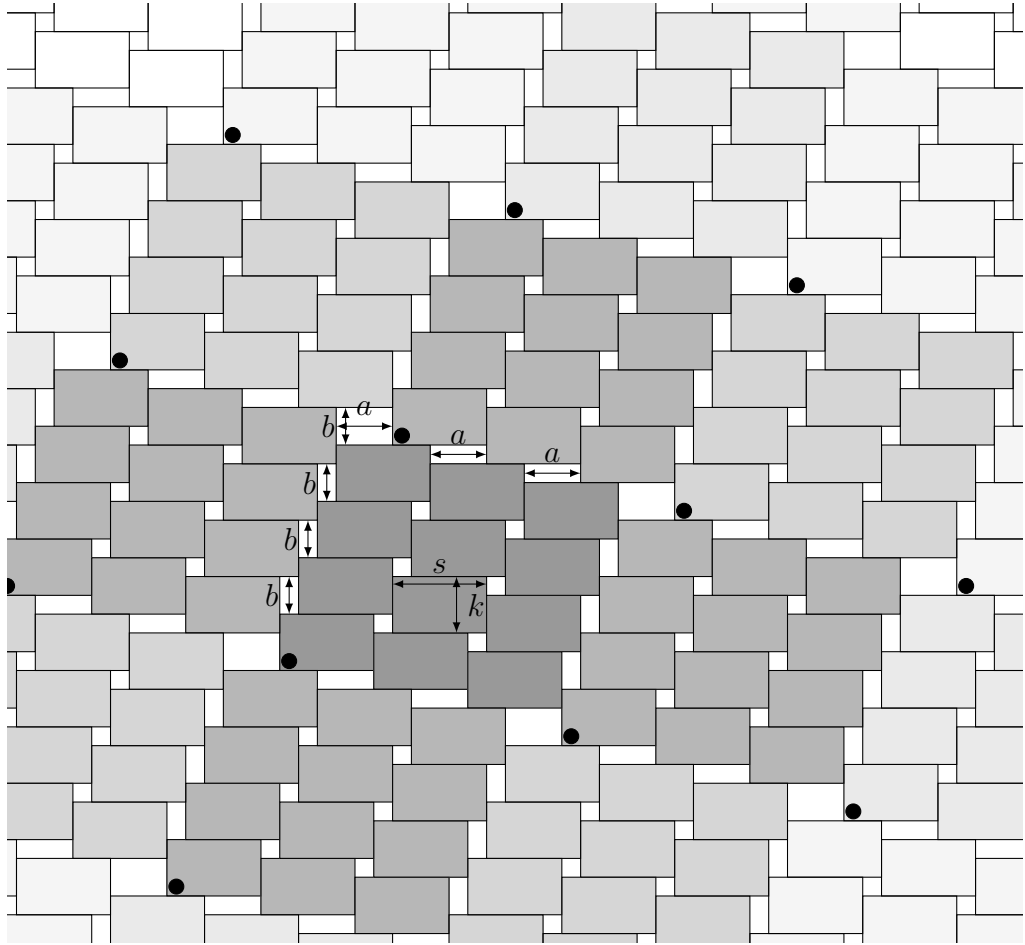
\begin{figure}
\begin{center}
\begin{tikzpicture}[fill=gray!80, scale=0.25,vertex/.style={circle,inner sep=2,fill=black,draw}]
\clip (-14.5,-14.5) rectangle (39.5,35.5);

\ar{0}{0}{1}
\ar{6}{12}{0.7}
\ar{15}{-4}{0.7}
\ar{-6}{-12}{0.7}
\ar{-15}{4}{0.7}
\ar{21}{8}{0.4}
\ar{-9}{16}{0.4}
\ar{-21}{-8}{0.4}
\ar{9}{-16}{0.4}
\ar{12}{24}{0.2}
\ar{27}{20}{0.1}
\ar{-3}{28}{0.1}
\ar{42}{16}{0}
\ar{-18}{32}{0}
\ar{-24}{20}{0.1}
\ar{-30}{8}{0.2}
\ar{-36}{-4}{0.1}
\ar{36}{4}{0.1}
\ar{30}{-8}{0.2}
\ar{24}{-20}{0.1}
\ar{18}{36}{0.1}
\ar{33}{32}{0}
\ar{3}{40}{0}
\ar{-12}{-24}{0.2}
\ar{-27}{-20}{0.1}
\draw [latex-latex] (0.5,3) -- (0.5,5) node[midway,left] {$b$};
\draw [latex-latex] (1.5,6) -- (1.5,8) node[midway,left] {$b$};
\draw [latex-latex] (2.5,9) -- (2.5,11) node[midway,left] {$b$};
\draw [latex-latex] (3.5,12) -- (3.5,14) node[midway,left] {$b$};
\draw [latex-latex] (3,13) -- (6,13) node[midway,above] {$a$};
\draw [latex-latex] (8,11.5) -- (11,11.5) node[midway,above] {$a$};
\draw [latex-latex] (13,10.5) -- (16,10.5) node[midway,above] {$a$};
\draw [latex-latex] (6,4.6) -- (11,4.6) node[midway,above] {$s$};
\draw [latex-latex] (9.4,2) -- (9.4,5) node[midway,right] {$k$};
\draw (0.5,0.5) node [vertex]{};
\draw (6.5,12.5) node [vertex]{};
\draw (12.5,24.5) node [vertex]{};
\draw (-5.5,-11.5) node [vertex]{};
\draw (15.5,-3.5) node [vertex]{};
\draw (21.5,8.5) node [vertex]{};
\draw (27.5,20.5) node [vertex]{};
\draw (10.5,-15.5) node [vertex]{};
\draw (-14.5,4.5) node [vertex]{};
\draw (-8.5,16.5) node [vertex]{};
\draw (-2.5,28.5) node [vertex]{};
\draw (30.5,-7.5) node [vertex]{};
\draw (36.5,4.5) node [vertex]{};
\end{tikzpicture}
\end{center}
\caption{A planar rectangle packing, with some blocks of rectangles highlighted}
\label{fig:rectanglesinplane}
\end{figure}

\section{More structure theory}
\label{sec:morestructure}

As in previous sections, we assume that $n$ and $k$ are fixed integers with $k<n$, and we set $s=\lfloor (n-1)/k\rfloor$ with $n=sk+r$ for $1\leq r\leq k$. We define $d_s$ and $d_k$ as before, and assume that $r<k$ and $r<s$. Futhermore, we suppose that $d_sd_k$ divides $n$, so $(s,k)$-clash-free permutations exist.
In this section, we sketch how the structural information in Section~\ref{sec:structure} may be extended, and the construction in Section~\ref{sec:construction} may be generalised, to provide a tighter characterisation of $(s,k)$-clash-free permutations.

The orbits of the map taking a permutation $\pi$ to the permutation $-\pi$ partition the set of $(s,k)$-clash-free permutations into pairs, one clockwise and one anticlockwise. So, without loss of generality, we may assume that our $(s,k)$-clash-free permutation is clockwise. Furthermore, replacing $\pi$ by the permutation $\pi'$ defined by $\pi'(x)=\pi(x)-\pi(0)$ for all $x\in\mathbb{Z}_n$, we may assume that $\pi(0)=0$. We aim to classify clockwise $(s,k)$-clash-free permutations with $\pi(0)=0$ by bringing them into bijection with the set of objects we call $(s,k,n)$-jumpers. It turns out that the rectangles of these clash-free permutations may be divided up into blocks, generalising the approach in Section~\ref{sec:construction}, but the gaps associated with the rectangles in a block have a more complicated structure that is determined by an $(s,k,n)$-jumper. We use the term jumper to keep within our weaving theme, and because a jumper determines how rectangles are arranged, jumping from a completely regular lattice. We leave the reader to think of a brand name that the jumpers might belong to.

An $(s,k,n)$-\emph{jumper} is a pair $\big((a_i),(b_i)\big)$ of sequences of integers with the following properties:
\begin{enumerate}
\item $(a_i)$ has period dividing $d_s$, and $(b_i)$ has period dividing $d_k$.
\item We have $1\leq a_i<s$ and $1\leq b_i<k$ for $i\geq 0$.
\item The $d_k$ partial sums $\sum_{i=0}^{\ell-1} b_i$  where $0\leq \ell<d_s$ are distinct modulo $d_k$. Moreover, $d_sd_k$ divides $\sigma_b$ where $\sigma_b=\sum_{i=0}^{d_k-1} b_i$. 
\item The $d_s$ partial sums $\sum_{i=0}^{m-1} a_i$  where $0\leq m<d_s$ are distinct modulo $d_s$. Moreover,  $d_sd_k$ divides $\sigma_a$ where $\sigma_a=\sum_{i=0}^{d_s-1} a_i$.
\item Defining $\sigma_a$ and $\sigma_b$ as above, $\sigma_a\sigma_b=d_sd_k r$.
\end{enumerate}
We define $J(s,k,n)$ to be the set of $(s,k,n)$-jumpers. 

For example, suppose $n=216$, $s=10$ and $k=21$. So $n=sk+r$ with $r=6$, and we have $d_s=2$ and $d_k=3$. We can check that $\big((a_i),(b_i)\big)$ is a jumper, Let $(a_i)=(a_0,a_1,\ldots)$ be the sequence of period $2$ with $a_0=1$ and $a_1=5$. Let $(a'_i)$ be the constant sequence of period $1$ where $a'_i=3$ for all $i$. Let $(b_i)=(b_0,b_1,\ldots)$ be the sequence of period $3$ where $b_0=b_1=1$ and $b_2=4$. Let $(b'_i)=(b'_0,b'_1,\ldots)$ be the sequence of period $3$ where $b'_0=1$, $b'_1=4$ and $b'_2=1$. Then is not hard to check that when $(c_i)\in\{(a_i),(a'_i)\}$ and $(d_i)\in\{(b_i),(b'_i)\}$ we have that $\big((c_i),(d_i)\big)$ is an $(s,k,n)$-jumper. 

Under our assumption that $d_sd_k$ divides $n$, $(s,k,n)$-jumpers always exist. Indeed, a $(s,k,n)$-jumper can be constructed as follows. Since $n=sk+r$, we see that $d_sd_k$ divides $r$. Write $r/(d_sd_k)=r_ar_b$ for some positive integers $r_a$ and $r_b$. (This factorisation can be trivial.) Define $\sigma_a=d_sd_kr_a$ and $\sigma_b=d_sd_kr_b$. Let $(a_0,a_1,\ldots)$ be the sequence of period dividing $d_s$ with $a_1=\cdots a_{d_s-1}=1$ and $a_0=\sigma_a-(d_s-1)$. Similarly, let $(b_0,b_1,\ldots)$ be the sequence of period dividing $d_k$ with $b_1=\cdots b_{d_k-1}=1$ and $b_0=\sigma_b-(d_k-1)$. It is easy to check (using our assumption that $r<s$ and $r<k$) that $\big((a_i),(b_i)\big)$ is an $(s,k,n)$-jumper.

We will see below that a jumper $\big((a_i),(b_i)\big)$ determines the size of gaps associated with rectangles in a block of a permutation: the $(u,v)$th rectangle in a block has a gap that is rectangular of width $a_v$ and height $b_u$ to its north-east. For example, a block of a permutation associated with the jumper $\big((a'_i),(b'_i)\big)$ is depicted in Figure~\ref{fig:jumper_example}.

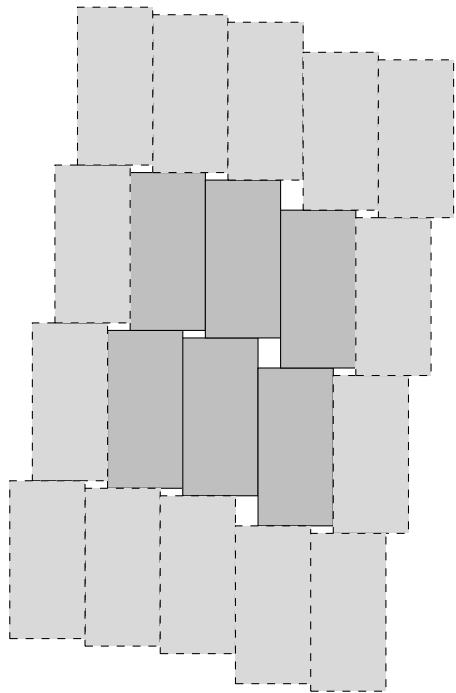
\begin{figure}
\begin{center}
\begin{tikzpicture}[fill=gray!50, scale=0.1,
vertex/.style={circle,inner sep=1.2,fill=black,draw}]
\filldraw (0,0) rectangle (10,21);
\filldraw (10,-1) rectangle (20,20);
\filldraw (20,-5) rectangle (30,+16);
\filldraw (3,21) rectangle (13,42);
\filldraw (13,20) rectangle (23,41);
\filldraw (23,16) rectangle (33,37);
\filldraw[dashed,fill=gray!30] (6,42) rectangle (16,63);
\filldraw[dashed,fill=gray!30] (16,41) rectangle (26,62);
\filldraw[dashed,fill=gray!30] (26,37) rectangle (36,58);
\filldraw[dashed,fill=gray!30] (30,-6) rectangle (40,15);
\filldraw[dashed,fill=gray!30] (33,15) rectangle (43,36);
\filldraw[dashed,fill=gray!30] (36,36) rectangle (46,57);
\filldraw[dashed,fill=gray!30] (-10,1) rectangle (0,22);
\filldraw[dashed,fill=gray!30] (-7,22) rectangle (3,43);
\filldraw[dashed,fill=gray!30] (-4,43) rectangle (6,64);
\filldraw[dashed,fill=gray!30] (-3,-21) rectangle (7,0);
\filldraw[dashed,fill=gray!30] (7,-22) rectangle (17,-1);
\filldraw[dashed,fill=gray!30] (17,-26) rectangle (27,-5);
\filldraw[dashed,fill=gray!30] (27,-27) rectangle (37,-6);
\filldraw[dashed,fill=gray!30] (-13,-20) rectangle (-3,1);
\end{tikzpicture}
\end{center}
\caption{A block in a permutation with jumper $\big((a'_i),(b'_i)\big)$. All gaps have width $3$, and height either $1$ or $4$.}
\label{fig:jumper_example}
\end{figure}

\begin{lemma}
\label{lem:coprime}
Let $n$ and $k$ be fixed integers with $k<n$. Set $s=\lfloor (n-1)/k\rfloor$, and define $r$ by $n=sk+r$ for $1\leq r\leq k$. Define $d_s=\gcd(n,s)$ and $d_k=\gcd(n,k)$. Assume that $d_sd_k$ divides $n$. Let $\big((a_i),(b_i)\big)$ be an $(s,k,n)$-jumper, and define $\sigma_a=\sum_{i=0}^{d_s-1} a_i$ and $\sigma_b=\sum_{i=0}^{d_k-1} b_i$. Then
\[
\gcd\big(n/(d_sd_k),\sigma_b/(d_sd_k)\big)=1.
\]
\end{lemma}
\begin{proof}
The third and fourth conditions for being a jumper imply that we be write $\sigma_a=d_sd_kr_a$ and $\sigma_b=d_sd_kr_b$ for some integers $r_a$ and $r_b$. The final condition of being a jumper implies that $r=d_sd_kr_ar_b$ and so in particular $r/(d_sd_k)$ is a multiple of $\sigma_b/(d_sd_k)$. Thus it suffices to show that 
\[
\gcd\big(n/(d_sd_k),r/(d_sd_k)\big)=1.
\]
Suppose, for a contradiction, that there exists a prime $p$ dividing both $n/(d_sd_k)$ and $r/(d_sd_k)$. Then $p$ divides\[
(n-r)/(d_sd_k)=sk/(d_sd_k)=(s/d_s)(k/d_k),
\]
and so $p$ divides $s/d_s$ or $k/d_k$. Suppose that $p$ divides $s/d_s$. Since $p$ divides $n/(d_sd_k)$, we see that $p$ divides $n/d_s$ and so $\gcd(n/d_s,s/d_s)\geq p$. But this contradicts the definition of $d_s$. Similarly, if we assume that $p$ divides $k/d_k$ we derive a contradiction using the definition of $d_k$. So $\gcd\big(n/(d_sd_k),r/(d_sd_k)\big)=1$ and the lemma follows. 
\end{proof}

\begin{theorem}
\label{thm:extremal_perms}
Let $n$ and $k$ be fixed integers with $k<n$. Set $s=\lfloor (n-1)/k\rfloor$, and define $r$ by $n=sk+r$ for $1\leq r\leq k$. Define $d_s=\gcd(n,s)$ and $d_k=\gcd(n,k)$. Assume that $r<k$ and $r<s$. Futhermore, suppose that $d_sd_k$ divides $n$. There is a bijection between the set of clockwise $(s,k)$-clash-free permutations with $\pi(0)=0$ and the set $J(s,k,n)$ of $(s,k,n)$-jumpers.
\end{theorem}
\begin{proof}
Let $\pi$ be a clockwise $(s,k)$-clash-free permutation with $\pi(0)=0$.  We produce an $(s,k,n)$-jumper as follows.

Consider the rectangle $R$ whose lower left cell is at $(0,0)$. As in the proof of Theorem~\ref{thm:gap}, we may define a sequence $U_0,U_1,\ldots$ of warp threads, where $U_0$ is the warp thread containing $R$ and where $U_{i+1}=\tau(U_i)$ for $i>0$. The proof of Theorem~\ref{thm:gap} shows that this sequence has period exactly $d_s$. Let $a_i$ be the width of a gap above a rectangle in $U_i$. (This width does not depend on the rectangle in $U_i$ we choose, by Lemma~\ref{lem:warpweft}.) The sequence $(a_i)$ has period dividing $d_s$, and $1\leq a_i<s$ for $i\geq 0$. Arguing similarly, by travelling along the warp thread through $R$, we may produce a sequence $V_0,V_1,\ldots$ of weft threads, where $V_0$ is the weft thread containing $R$ and $V_{i+1}=\rho(V_i)$ for $i>0$. This sequence has period exactly $d_k$. Let $b_i$ be the height of a gap above a rectangle in $V_i$ (which is well-defined by Lemma~\ref{lem:warpweft}). The sequence $(b_i)$ has period dividing $d_k$, and $1\leq b_i<k$ for $i\geq 0$. We claim that the pair $\big((a_i),(b_i)\big)$ of sequences is an $(s,k,n)$-jumper.

We have already established the first two conditions required for being a jumper. To show the third and fourth conditions are also satisfied, we argue as follows. Moving along the warp thread containing $R$, we find a sequence $R=R_{(0)},R_{(1)},\ldots $ of rectangles. The lower left cell of rectangle $R_{(\ell)}$ is at position $(\ell s,-\sum_{i=0}^{\ell-1} b_i)$. The rectangles $R_{(0)},R_{(1)},\ldots ,R_{(d_k-1)}$ lie in distinct weft threads $V_0,V_1,\ldots,V_{d_k-1}$, and so the elements $-\sum_{i=0}^{\ell-1} b_i$ are distinct modulo $d_k$, by the last statement of Lemma~\ref{lem:threadsize}. Hence the $d_k$ partial sums $\sum_{i=0}^{x-1} b_i$ are distinct modulo $d_k$. 
The sequence $R=R_{(0)},R_{(1)},\ldots $ of rectangles has period $n/d_s$, again by Lemma~\ref{lem:threadsize}. Comparing the second coordinates of $R_{(0)}$ and $R_{(n/d_k)}$ we find that
\[
0\equiv-\sum_{i=0}^{n/d_s-1}b_i\equiv -(n/(d_sd_k))\sigma_b\bmod n,
\]
and so $(n/(d_sd_k))\sigma_b\equiv 0\bmod d_sd_k$.
Thus the third condition holds. The fourth condition similarly holds, by arguing that the rectangles encountered when moving along the weft thread containing $R$ have lower left corners at positions $(\sum_{i=0}^{m-1} a_i,mk)$ and using the second statement of Lemma~\ref{lem:threadsize}.

It remains to check that the final condition for being a jumper is satisfied. Since the sequence $(b_j)$ has period $d_k$, the gaps of rectangles in warp thread $U_i$ have total area
\[
\sum_{j=0}^{n/d_s-1} (a_ib_j)=a_i\left(\sum_{j=0}^{n/d_s-1} b_j\right)=a_i(n/d_sd_k)\sigma_b.
\]
So the total area of all gaps can be written as
\[
(n/d_sd_k)\sum_{i=0}^{d_s-1}a_i\sigma_b=(n/d_sd_k)\sigma_a\sigma_b.
\]
But Lemma~\ref{lem:rowcolumnfree} shows that every row of our torus contains $r$ free cells, and so there are $nr$ free cells in total. Hence $nr= (n/d_sd_k)\sigma_a\sigma_b$ and so $\sigma_a\sigma_b=rd_sd_k$. This establishes our claim that $((a_i),(b_i))$ is an $(s,k,n)$-jumper, as required.

Let $\big((a_i),(b_i)\big)$ be an $(s,k,n)$-jumper. We construct a permutation $\pi:\mathbb{Z}_n\rightarrow\mathbb{Z}_n$ as follows. Let $x\in\mathbb{Z}_n$. 
By the fourth condition of being an $(s,k,n)$-jumper, and by considering $x$ modulo $d_s$, there exist unique integers $q\in\{0,1,\ldots,n/d_s-1\}$ and  $m\in\{0,1,\ldots,d_s-1\}$ such that $x=qd_s+\sum_{i=0}^{m-1}a_i\bmod n$. As in the proof of Theorem~\ref{thm:existence}, we may write $qd_s=fs\bmod n$ for a unique integer $f\in\{0,1,\ldots,n/d_s-1\}$ and we may write $f=\alpha d_k+\ell$ where $0\leq \ell<d_k$ and $0\leq \alpha<n/(d_sd_k)$. We define $\pi(x)=y$ where
\[
\pi(x)=\pi\left((\alpha d_k+\ell)s+\sum_{i=0}^{m-1}a_i\right)= mk-\sum_{i=0}^{\ell-1} b_i-\alpha \sigma_b.
\]
In order to check that $\pi$ is a permutation, it suffices to show that $\pi$ is injective. Suppose that
\[
mk-\sum_{i=0}^{\ell-1} b_i-\alpha \sigma_b=m'k-\sum_{i=0}^{\ell'-1} b_i-\alpha' \sigma_b\bmod n.
\]
Considering this equality modulo $d_k$, and using the third condition of being an $(s,k,n)$-jumper, we see that $\ell=\ell'$. So
$mk-\alpha \sigma_b=m'k-\alpha' \sigma_b\bmod n$ and hence (using the fact that $d_k$ divides $\sigma_b$)
\[
m(k/d_k)-\alpha (\sigma_b/d_k)=m'(k/d_k)-\alpha' (\sigma_b/d_k)\bmod n/d_k.
\]
Since $d_s$ divides both $n/d_k$ and $\sigma_b/d_k$, we deduce that $m(k/d_k)=m'(k/d_k)\bmod d_s$. But $k/d_k$ is coprime to $n/d_k$ and so $k/d_k$ is coprime to $d_s$. Hence $m=m'\bmod d_s$ and so $m=m'$. We deduce that $\alpha\sigma_b=\alpha'\sigma_b\bmod n$, and so $\alpha\sigma_b/(d_sd_k)=\alpha'\sigma_b/(d_sd_k)\bmod n/(d_sd_k)$. Now $\sigma_b/(d_sd_k)$ is coprime to $n/(d_sd_k)$, by Lemma~\ref{lem:coprime}, and so $\alpha=\alpha'$.

We have shown that $\pi$ is a permutation, it is clear that $\pi(0)=0$, and it is not hard to see that $\pi$ is clockwise. Consider the generalisation of the packing of $s\times k$ rectangles in the plane $\mathbb{Z}^2$ mentioned in Section~\ref{sec:construction}. So the $(\ell,m)$th rectangle in the packing has lower left corner at the cell with position
\[
\left(\ell s+\sum_{i=0}^{m-1}a_i,m k-\sum_{i=0}^{\ell-1}b_i\right).
\]
This rectangle packing is invariant with respect to the integer lattice $\cal L$ generated by $u$ and $v$, where
\begin{align*}
u&=(d_k s,-\sum_{i=0}^{d_k-1} b_i)=(d_k s,-\sigma_b)\text{ and}\\
v&=(\sum_{i=0}^{d_s-1}a_i,d_s k)=(\sigma_a,d_s k).
\end{align*}
We note that
\begin{align*}
\frac{k}{d_k}u+\frac{\sigma_b}{d_sd_k}v&=(ks+r,0)=(n,0)\text{ and}\\
-\frac{\sigma_a}{d_sd_k}u+\frac{s}{d_s}v&=(0,r+ks)=(0,n).
\end{align*}
Now $d_sd_k$ divides $\sigma_a$ and $\sigma_b$, since we have an $(n,s,k)$-jumper. So $(n,0)$ and $(0,n)$ are $\mathbb{Z}$-linear combinations of $u$ and $v$, and therefore lie in $\cal L$. Thus our packing induces a well-defined packing of the torus $\mathbb{Z}_n^2$. The argument in Theorem~\ref{thm:existence} now shows that $\pi$ is $(s,k)$-clash-free.

The process of producing a jumper from a permutation, and the process of producing a permutation from a jumper, are inverse to each other. So we have a bijections between $(s,k,n)$-jumpers and clockwise $(s,k)$-clash free permutations of $\mathbb{Z}_n$ with $\pi(0)=0$, as required.
\end{proof}

\end{document}